\documentclass[11pt, a4paper]{amsart}
\usepackage{times}
\usepackage{a4wide}
\usepackage{hyperref}
\usepackage{enumerate}
\usepackage{amsmath, amscd, amsfonts, amsthm, amssymb,color}
\usepackage[all]{xypic}
\usepackage[T1]{fontenc}
\usepackage[english]{babel}

\newtheorem{theorem}{Theorem}[section]
\newtheorem{corollary}[theorem]{Corollary}
\newtheorem{lemma}[theorem]{Lemma}

\newtheorem{proposition}[theorem]{Proposition}

\newcommand{\Asai}{{\rm Asai}}

\newcommand{\Hom}{{\rm Hom}}

\newcommand{\sgn}{{\rm sgn}}

\newcommand{\GL}{\mathrm{GL}}

\newcommand{\SL}{\mathrm{SL}}

\newcommand{\proj}{\mathrm{proj}}

\newcommand{\st}{\mathrm{st}}
\newcommand{\tg}{\mathrm{tg}}

\DeclareMathOperator{\Gal}{Gal}

\DeclareMathOperator{\Frob}{Frob}

\newcommand{\PGL}{\mathrm{PGL}}
\newcommand{\PSL}{\mathrm{PSL}}

\newcommand{\donothing}[1]{}

\newcommand{\mat}[4]{
 \left(  \begin{smallmatrix} #1 & #2 \\ #3 & #4 \end{smallmatrix} \right)}

\newcommand{\cG}{\mathcal{G}}

\newcommand{\fp}{\mathfrak{p}}

\newcommand{\AAA}{\mathbb{A}}

\newcommand{\CC}{\mathbb{C}}

\newcommand{\FF}{\mathbb{F}}

\newcommand{\PP}{\mathbb{P}}
\newcommand{\QQ}{\mathbb{Q}}

\newcommand{\ZZ}{\mathbb{Z}}

\newcommand{\Qbar}{\overline{\QQ}}

\newcommand{\Fbar}{\overline{\FF}}

\newcommand{\Gbar}{\overline{G}}

\newcommand{\rhotilde}{\widetilde{\rho}}

\newcommand{\gb}{{\overline{g}}}
\newcommand{\hb}{{\overline{h}}}
\newcommand{\ib}{{\overline{i}}}

\newcommand{\res}{\mathrm{res}}

\newcommand{\infl}{\mathrm{infl}}

\newcommand{\tr}{\mathrm{tr}}

\DeclareMathOperator{\h}{H}
\DeclareMathOperator{\z}{Z}

\newcommand{\exclude}[1]{}

\begin{document}
\title[Splitting fields, prime decomposition and modular forms]{Splitting fields of  $X^n-X-1$ (particularly for $n=5$), prime decomposition and modular forms}
\author{Chandrashekhar B.\ Khare, Alfio Fabio La Rosa, Gabor Wiese}

\address[C.\ B.\ Khare]{Department of Mathematics, University of California, Los Angeles, CA 90095, U.S.A.}
\email{shekhar@math.ucla.edu}

\address[A.\ F.\ La Rosa, G.\ Wiese]{Department of Mathematics, University of Luxembourg, Maison du nombre, 6~avenue de la Fonte, L-4364 Esch-sur-Alzette, Luxembourg}
\email{fabio.larosa@uni.lu, gabor.wiese@uni.lu}

\dedicatory{Dedicated to the memory of Bas Edixhoven.}

\maketitle

\begin{abstract}
We study the splitting fields of the  family of  polynomials $f_n(X)= X^n-X-1$.  This family of polynomials has been much studied in the literature and has some remarkable properties. In \cite{Jordan},  Serre related the function on primes $N_p(f_n)$, for a fixed $n \leq 4$ and $p$ a varying prime,  which counts the number of roots of $f_n(X)$ in $\mathbb F_p$ to coefficients of modular  forms. We study the case  $n=5$, and relate $N_p(f_5)$ to mod $5$  modular   forms over $\mathbb Q$, and  to characteristic 0,  parallel weight 1 Hilbert modular forms over $\mathbb Q(\sqrt{19 \cdot 151})$. 
\end{abstract}

\section{Introduction}

Serre in \cite{Jordan} considers the family of polynomials $f_n(X)= X^n-X-1 \in \ZZ[X]$ for integers $n \geq 2$.

This is a fascinating family in a number of ways.
The irreducibility of $f_n$ was established by Selmer in~\cite{Selmer}.
The discriminant of $f_n$ equals $d_{f_n} = (-1)^{(n-1)(n-2)/2} \cdot  (n^n - (1-n)^{n-1})$.
The first remarkable fact is that the discriminant of the associated number field $K_{f_n} = \QQ[X]/(f_n(X))$ is squarefree and that the residue degree of any ramified prime is~$1$.\footnote{This follows immediately from \cite[Theorems 1,2]{LNV2}. It can also be proved directly as follows. Let $p$ be a prime dividing $d_{f_n}$. Then $p \nmid n(n-1)$, in particular $p \neq 2$. Suppose $a \in \Fbar_p$ is a multiple root of~$f_n(X)$ modulo~$p$. Then it is also a root of $f_n'(X) = n X^{n-1}-1$ modulo~$p$, from which $a=n/(1-n)$ follows. As $f_n''(a) \neq 0$ we see that $f_n(X)$ factors as $f_n(X) \equiv (X-\frac{n}{1-n})^2 \cdot g(X) \mod{p}$ with $g(X)$ being squarefree and coprime to $(X-\frac{n}{1-n})$ modulo~$p$. From this it follows that there is at most one ramified prime above~$p$ and then that prime has residue degree~$1$ and ramification index~$2$. Consequently, the discriminant of the number field is squarefree.}
By work of Kondo~\cite{Kondo} (Theorems 1 and~2) this implies that the Galois group of $f_n(X)$, {\it i.e.} the Galois group of the Galois closure $K$ of $K_{f_n}$ over $\QQ$, is the symmetric group~$S_n$, that $K$ contains the quadratic field $E=\QQ(\sqrt{d_{f_n}})$ and that the extension $K/E$ is unramified at all finite primes.
It should be stressed that constructing polynomials (and number fields) with squarefree discriminants is a hard problem. Kedlaya~\cite{Kedlaya} gave a construction, whose crucial point is that the signature of the field can be prescribed.
In our case, the signature of $K_{f_n}$ is $(1,n-1)$ if $n$ is odd\footnote{Indeed, $f_n$ has exactly one real root if $n$ is odd because $f_n(x) < 0$ for all $x \le 1$ and it does not possess any local extremum at any $x \ge 1$.} and $(2,n-2)$ if $n$ is even.\footnote{Indeed, if $n$ is even, then $f_n$ has exactly two real roots because it has a unique local minimum and takes a negative value there.}
Consequently, the image of any complex conjugation of $K/\QQ$ in $S_n$ is a product of $\lfloor \frac{n-1}{2} \rfloor$ transpositions with disjoint supports and $E$ is real if and only if $n$ is congruent to $1$ or $2$ modulo~$4$.

It might be useful to mention that the polynomials $f_n$ themselves do not always have a squarefree discriminant. These have been studied on their own in a number of papers, for instance in~\cite{BMT}. The first $n$ for which $d_{f_n}$ is not squarefree is $130$.\footnote{See \url{https://oeis.org/A238194}}
It is important to point out that the squarefreeness of $d_{f_n}$ implies that $\ZZ[X]/(f_n(X))$ is the ring of integers of~$K_{f_n}$ because in general the square of the index divides discriminant. We will exploit this for $n=5$.

In \cite{Jordan}, Serre asks for information about the number of roots $N_p(f_n)$ of these  polynomials in $\FF_p$. In other words he considers the point counting function $\#A_{f_n}(\FF_p)$  where  $A_{f_n}$ is the zero-dimensional affine scheme $\ZZ[X]/(f_n(X))$.
We let $\theta_n^\st$ be the $n-1$-dimensional standard representation defined as the quotient of the natural $n$-dimensional permutation representation of~$S_n$ modulo the trivial one. Then we have for all unramified primes~$p$
\[N_p(f) = \tr(\theta_n^\st(\Frob_p)) + 1.\]
For $n \leq 4$, Serre relates the $N_p(f_n)$’s to coefficients of modular forms, as follows.

\subsection*{$n=3$} (See~\cite[\S5.3 and corresponding notes]{Jordan}.) For $n=3$ the splitting field of $f_3$ is an $S_3$-extension $K$ of $\QQ$ and $\zeta_{K_{f_3}}(s)/\zeta(s)$ is a holomorphic function, equal to the Artin $L$-function $L(\theta_3^\st,s)$. The representation $\theta_3^\st$ arises by induction of an order 3 character of $\Gal(K/E)$ where $E$ is $\QQ(\sqrt{-23})$. It turns  out that $K$ is the Hilbert class field of $E$, and in fact is the maximal unramified extension of $E$.
Thus the representation $\theta_3^\st$  arises from the weight one (dihedral or CM) modular form $F(z) \in S_1(\Gamma_0(23), (\frac{\cdot}{23}))$ given by the product formula (with   $q=e^{2\pi iz}$) 
\[F(z)=q\cdot \prod_{k=1}^\infty  (1-q^k) \cdot \prod_{k=1}^\infty   (1-q^{23k}).\] It can also be written in terms of $\Theta$-series as 
\[\frac{1}{2} \sum_{x,y \in \mathbb Z}  \ \ q^{x^2+xy+6y^2} - \frac{1}{2} \sum_{x,y \in \mathbb Z} \ \  q^{2x^2+6xy+3y^2}.\]

The $S_3$-extension $K/\mathbb Q$  is also the field cut out by the Galois representation associated to the Ramanujan $\Delta$-function mod~$23$. This is explained by the congruence
\[ \Delta(q) = q \cdot \prod_{k=1}^\infty (1-q^k)^{24} \equiv q \cdot \prod_{k=1}^\infty  (1-q^k) \cdot \prod_{k=1}^\infty   (1-q^{23k}) \pmod{23}.\]

\subsection*{$n=4$} (See~\cite[\S5.4 and corresponding notes]{Jordan}.) For degree $n=4$ the picture to relate $N_p(f_4)$ to a weight one modular form is more complicated. Serre observes that:
\begin{enumerate}[(i)]
\item The $S_4=\PGL_2(\FF_3)$-extension is embedded in $\tilde K=K(\sqrt{7-4x^2})$, where $K$ is the splitting field of~$f_4$ over~$\QQ$ and $x$ is a root of $f_4(X)$,  of degree 2 over $K$ and the resulting Galois group is isomorphic to $\GL_2(\FF_3)$, which embeds in $\GL_2(\ZZ[\sqrt{-2}])$.  Furthermore   the field $\tilde K$ turns out to be the maximal unramified extension of the quadratic field $E=\QQ(\sqrt{-283})$.

\item From this one gets a 2-dimensional odd representation $\rho:G_\QQ \to GL_2(\CC)$ with fixed field $\tilde K$ such that
\[\rho \otimes \rho=\epsilon\oplus \theta_4^\st,\]
factoring through a representation of $\Gal(K/\QQ) = S_4$, where $\epsilon$ is the sign character of a permutation.
The representation $\rho$ arises by the Deligne-Serre construction from a modular form $f \in S_1(\Gamma_0(283),\chi)$ with $\chi$ the order 2 Dirichlet character of conductor 283. This also gives  $N_p(f)-1= a_p(f)^2- \epsilon(p)$.
\end{enumerate}

\subsection*{$n=5$} 
In this article, we look at the $n=5$ case of the question and try to respond to the gauntlet implicitly thrown down by Serre:
\begin{quote}
{\it The case $n \geq 5$. Here the only known result seems to be that  $f_n = x^n -x-1$  is irreducible (Selmer [15]) and that its Galois group is the symmetric group $S_n$. No explicit connection with modular forms (or modular representations) is known, although some must exist because of the Langlands program.}
\end{quote}

The paper  partly arose from graduate courses one of us (CBK) has taught at UCLA partly  based on Serre’s paper \cite{Jordan} in which he suggested that students try and tackle Serre’s challenge. 

We propose two approaches to answer Serre's challenge.
The first one is computational in nature. The key point is that we explore the exceptional isomorphism $S_5 \cong \PGL_2(\FF_5)$ and, instead of $K_{f_5}$, work with a degree~$6$ polynomial~$g \in \QQ[X]$ on which the Galois group acts via the natural action of $\PGL_2(\FF_5)$ on the projective line $\PP^1(\FF_5)$, leading to an `exotic' embedding of $S_5$ into~$S_6$.
Serre pointed out to us that in classical terms, this was called the {\em sextic resolvent} and that this `exotic' $S_5$ may also be viewed as the image of the standard $S_5$ under an outer automorphism of $S_6$.

Explicit class field theory then allows us to explicitly solve a Galois embedding problem for the group $C_{4}\cdot_{6} S_{5} = 4_{-}\PGL_2(\FF_5)$ (in notation of Tim Dokchiter's project GroupNames~\cite{GroupNames} and Quer's article~\cite{Quer}, respectively). That group can be characterised as the unique central extension of $\PGL_2(\FF_5)$ by $C_4$ which restricts to the unique non-trivial central extension of $\PSL_2(\FF_5)$ by $C_4$ and which is not $\GL_2(\FF_5)$. We solve the embedding problem by computing a polynomial $h \in \QQ[X]$ of degree~$48$ describing a cyclic $C_8$ extension of $K_g = \QQ[X]/(g(X))$. It corresponds to a subgroup of $C_{4}\cdot_{6} S_{5}$ which is isomorphic to the dihedral group~$D_5$.
An explicit polynomial $h(X)$ is included in~\S\ref{subsec:h}.

This leads to the following result, linking $f_5$ to a modular form~$F$ of weight one in characteristic~$5$. Its Hecke eigenvalues at primes $p \neq 19,151$ can be explicitly computed from the polynomial~$h(X)$, for instance, using Magma~\cite{Magma}.

\begin{theorem}\label{thm:F}
There is a Hecke eigenform $F$ in $S_{1}(19\cdot 151^{2}, \chi_{-19},\Fbar_5)$, the space of weight one cuspidal Katz modular forms\footnote{Katz modular forms were defined by Katz in~\cite{KatzPadicProp}. A comprehensive account of them is given in~\cite{Goren} and a short summary can be found in~\cite{EdixSC}.} of level $19\cdot 151^{2}$ and Dirichlet character $\chi_{-19}$ corresponding to $\QQ(\sqrt{-19})/\QQ$ enjoying the following properties.
The attached Galois representation
\[\rho:G_\QQ:= \Gal(\Qbar/\QQ) \rightarrow \GL_{2}(\Fbar_5)\]
has conductor $19\cdot 151^{2}$, its image is isomorphic to $C_{4}\cdot_{6} S_{5} = 4_{-}\PGL_2(\FF_5)$, its projectivisation
\[\rho^\proj:G_\QQ \xrightarrow{\rho} \GL_{2}(\Fbar_5) \to \PGL_2(\Fbar_5)\]
has image $\PGL_{2}(\FF_{5})\cong S_{5}$ and $\ker(\rho^\proj)$ is the absolute Galois group of~$K$, the splitting field of~$f_5$.
The modular form $F$ is not the reduction of any holomorphic modular form of weight one in any level.

Moreover, the restriction of $\rho$ to the absolute Galois group $G_E$ of $E=\QQ(\sqrt{19 \cdot 151})$ is unramified at~$19$, but ramifies at $151$. However, there is a character $\delta$ of~$G_E$ such that $\rho|_{G_E}\otimes \delta$ is unramified at all finite places.
\end{theorem}

The second approach consists in showing the existence of a cuspidal Hilbert eigenform of parallel weight $1$ defined over the real quadratic field $E$, denoted by~$G$, such that a twist of its Asai transfer to $\GL_{4}(\AAA_\QQ)$ is attached to an Artin representation which factors through~$\theta_5^\st$.
Let us record for completeness that the conductor of $\theta_5^\st$ equals $19\cdot 151$.\footnote{This follows because the only ramified primes in $K/\QQ$ are $19$ and $151$ and, as $K/E$ is unramified at all finite places, the inertia groups at both these primes are cyclic of order~$2$. Moreover, they are generated by a transposition since they are not contained in~$A_5$. As a consequence they fix a $3$-dimensional subspace of vector space underlying~$\theta_5^\st$, leading to the claimed conductor.}
This uniquely identifies the representation as \cite[Artin representation 4.2869.5t5.b.a]{lmfdb}. 

We prove the following general result, which slightly improves on a theorem of Calegari~\cite{Cal}. The improvement results from using stronger modularity results due to Pilloni and Stroh~\cite{Pil} than were available when \cite{Cal} was written.

\begin{theorem}\label{thm:Asai}
Let $K/\QQ$ be a Galois extension with Galois group $\Gal(K/\QQ)\cong S_{5}$ and let $E$ be the subfield of $K$ fixed by the subgroup of $\Gal(K/\QQ)$ isomorphic to $A_{5}$. We assume that $K$ is totally imaginary and $E$ is real.

\begin{enumerate}[(a)]
\item \label{thm:Asai:a} There is a Hilbert modular eigenform $G$ over~$E$ of parallel weight one with attached Galois representation $\eta:G_E \to \GL_2(\CC)$ and there is a character $\chi: G_\QQ \to \CC^\times$ such that
\[ \theta_5^\st \cong \Asai_{G_E}^{G_\QQ}(\eta) \otimes \chi.\]

\item \label{thm:Asai:b} Let $\Pi=\Pi_{\infty}\otimes \bigotimes_{p}\Pi_{p}$ be the automorphic form on $\GL_{4}(\AAA_\QQ)$ obtained by twisting the Asai transfer of $G$ from $E$ to~$\QQ$ by the Hecke character corresponding to~$\chi$. Then the L-function of~$\Pi$ coincides with that of $\theta_5^\st$ at all but finitely many primes~$p$:
\[ L_{p}(s,\Pi)^{-1}=\det(I-\theta_5^\st(\Frob_{p})p^{-s}). \]
\end{enumerate}
\end{theorem}

We remark that Calegari's original theorem depended on the local condition that the Frobenius element at $5$ should be conjugate to the double transposition $(1,2)(3,4)$. We note that this is not satisfied if $K$ is the splitting field of $f_{5}$ because $f_5$ is irreducible modulo~$5$, whence the Frobenius element is a $5$-cycle. We further point out that it was also remarked in \cite{Dwyer} that this condition is superfluous.

\medskip
We now elaborate more on how our results respond to Serre's challenge, which we interpret in various ways:
\begin{enumerate}[(1)]
\item\label{int:1} Vaguely, we see it as asking about a relation between the polynomial~$f_{5}$ and `automorphic objects'.
\item\label{int:2} More precisely, we see it as relating conjugacy classes of unramified Frobenius elements in $\Gal(K/\QQ)$ to Hecke eigenvalues of `automorphic objects'.
\item\label{int:3} We also consider the original question about expressing $N_p(f_5)$ explicitly via an automorphic form.
\end{enumerate}

The vague interpretation \eqref{int:1} can be answered affirmatively in a number of ways and in different characteristics.
For instance, in characteristic~$5$, the projective Galois representation attached to the weight one form $F$ over~$\Fbar_5$ cuts out the field $K/\QQ$. This is also true for the Hilbert modular form~$G$ over~$E$ in characteristic~$0$.
One can also reduce $G$ modulo (a prime above)~$2$. The image of the attached Galois representation will be $\SL_2(\FF_4) \cong A_5$ and $K$ will again be the field cut out by it.
In that sense, $F$, $G$ and the reduction of $G$ modulo~$2$ are automorphic objects giving rise to~$K/\QQ$, and their Galois representations `control' $K/\QQ$ and, hence, the arithmetic of the polynomial~$f_{5}$.

Concerning the more precise interpretation~\eqref{int:2}, from Theorem~\ref{thm:F} we obtain the following result.

\begin{corollary}\label{cor:int2}
Let $F$ be the weight one Katz modular eigenform over~$\Fbar_5$ from Theorem~\ref{thm:F}.
For any prime~$p$, denote by $a_p(F)$ the eigenvalue of the Hecke operator $T_p$ on~$F$.

Then for every prime $p \nmid 19 \cdot 151$, we have the formula
\[ N_p(f_5) \equiv 1 + \left( \frac{19 \cdot 151}{p} \right) \cdot \left( \left( \frac{-19}{p} \right) a_p(F)^2 + \left( \frac{19 \cdot 151}{p} \right) - 1 \right) \bmod{5}.\]

Moreover, for every prime $p \nmid 19 \cdot 151$, the triple
$(a_p(F),\left( \frac{-19}{p} \right), \left( \frac{-151}{p} \right))$
uniquely determines the conjugacy class of any Frobenius element at~$p$ in $\Gal(K/\QQ) \cong S_5$ with the exception that $5$-cycles cannot be distinguished from the identity.
\end{corollary}

The exception is due to the fact that order~$5$ elements are non-semisimple in characteristic~$5$.
Since the determinant of $\rho$ is $\chi_{-19} = \left( \frac{-19}{\cdot} \right)$, we recognise that 
$\left( \frac{-19}{p} \right) a_p(F)^2 = \frac{a_p(F)^2}{\left( \frac{-19}{p} \right)}$ indeed only depends on the projective representation $\rho^\proj$.

\medskip
Concerning \eqref{int:3}, Serre's original question is completely answered by the following corollary of Theorem~\ref{thm:Asai}, which also gives the strongest form of~\eqref{int:2}.

\begin{corollary}\label{cor:Asai}
Let $K$ be the splitting field of~$f_5$ over~$\QQ$.
There exists a Hilbert modular eigenform~$G$ over~$E=\QQ(\sqrt{19 \cdot 151})$ of parallel weight one, the $T_\fp$-eigenvalues of which are denoted $a_\fp(G)$, such that for every prime~$p \nmid 19 \cdot 151$, we have the formula
\[ N_p(f_5) = 1 + \left(\frac{-151}{p}\right) \cdot \prod_{\fp \mid p} a_{\fp}(G),\]
where $\fp$ runs through the primes of~$E$ above~$p$.

Moreover, for every prime $p \nmid 19 \cdot 151$, the triple $((a_\fp(G))_{\fp \mid p},\left( \frac{-19}{p} \right), \left( \frac{-151}{p} \right))$ uniquely determines the conjugacy class of any Frobenius element at~$p$ in $\Gal(K/\QQ) \cong S_5$.
\end{corollary}

We remark that holomorphic cuspidal Hilbert modular forms over~$E$ of parallel weight~$1$ and level~$1$ were constructed by Bryk, who used them to express the square of $N_p(f_5)$ (see \cite[Prop.~6.3.5]{Bryk}).
Bryk's forms are different from the modular form~$G$ of Corollary~\ref{cor:Asai} because the level of~$G$ is non-trivial. However, the Artin representation attached to any of Bryk's forms cuts out an everywhere unramified extension of~$E$. Its projectivisation is the same as that of $G$, namely the one cutting out $K/E$.
This implies that our form $G$ is a twist of one of Bryk's forms by a Hecke character of~$E$.

Finally, as we remarked at the beginning of the introduction, the splitting field $K$ of~$f_n$ is an $A_n$-extension of its quadratic subfield~$E$ that is unramified at all finite places. For $n=3$ and $n=4$, one knows the maximal extension $L/E$ that is unramified at all finite places explicitly.
In both these cases, $L$ is the field cut out by a linear Galois representation with projectivisation corresponding to the $A_n$-extension $K/E$.
This pattern does not continue for $n=5$. Indeed, $K_f(\sqrt{-151})$ has class number~$7$ and $EK_f(\sqrt{-151})$ is an unramified extension of~$E$; its class number is hence divisible by~$7$ (in fact, it equals $21$ under GRH). So, there is a cyclic extension of~$K$ of order~$7$, which is unramified over~$E$, but not accounted for by the Artin representations in question because the orders of their images are coprime to~$7$.

\subsection*{Acknowledgements}
We would like to thank the anonymous referee for useful suggestions improving the exposition of the paper.
We also thank J-P.~Serre for helpful remarks on the manuscript. We also thank J\"urgen Kl\"uners for significantly reducing the size of the polynomial~$h(X)$.
G.W.\ thanks Sara Arias-de-Reyna for very helpful suggestions, Jordi Guàrdia and Enric Nart for helpful communication about maximal orders as well as Frank Calegari for recalling the article \cite{Dwyer} to us.

\section{Certain central extensions and embedding problems}\label{sec:gp}

We start with a very brief outline of the general theory relating lifts of projective representations, Galois embedding problems and group cohomology.

Let $\pi: G \twoheadrightarrow \Gbar$ be a surjective group homomorphism and assume $C=\ker(\pi)$ lies in the centre of~$G$.
We are interested in lifting a {\em projective representation}
\[ \rho^\proj: \Gbar \to \PGL_n(K) \]
to a linear representation $\rho: G \to \GL_n(K)$, where $K$ is any field and $n \in \ZZ_{\ge 1}$.

We will use (continuous) cochains in (continuous) group cohomology, which are for instance explained in~\cite[Chapter~1]{NSW}.
We consider the multiplicative group $K^\times$ as a trivial module for $G$ and~$\Gbar$ and start by associating with $\rho^\proj$ an inhomogeneous $2$-cocycle $\gamma \in \z^2(\Gbar,K^\times)$ as follows:
for each $\gb \in \Gbar$, choose once and for all a lift $\rhotilde(\gb) \in \GL_n(K)$ of $\rho^\proj(\gb)$ such that $\rhotilde(1)=1$ and let
\[ \gamma(\gb,\hb) := \rhotilde(\gb) \cdot \rhotilde(\hb) \cdot  \rhotilde(\gb\hb)^{-1}.\]
One easily checks that $\gamma$ is indeed an inhomogeneous $2$-cocycle, {\it i.e.\ } that it satisfies
$\gamma(\gb\hb,\ib)  \cdot \gamma(\gb,\hb) = \gamma(\hb,\ib) \cdot  \gamma(\gb,\hb\ib)$ for all $\gb,\hb,\ib \in \Gbar$.
Via inflation along~$\pi$ ({\it i.e.\ } by precomposing with~$\pi$), we also consider $\gamma$ as a $2$-cocycle of~$G$, {\it i.e.\ } as an element of $\z^2(G,K^\times)$.
By definition, the inflation of~$\gamma$ is a $2$-coboundary if and only if there is a map of sets $\sigma: G \to K^\times$ such that $\gamma(g,h) = \sigma(g) \cdot \sigma(h) \cdot \sigma(gh)^{-1}$ for $g,h \in G$.
If this is the case, by equating the two expressions for $\gamma$, we have for all $g,h \in G$:
\[\big(\rhotilde(\pi(g))\sigma(g)^{-1}\big) \cdot \big(\rhotilde(\pi(h))\sigma(h)^{-1}\big) = \big(\rhotilde(\pi(gh))\sigma(gh)^{-1}\big),\]
showing that $\rho:G \to \GL_n(K)$ with $\rho(g) = \rhotilde(\pi(g))\sigma(g)^{-1}$ is a linear representation of~$G$ lifting $\rho^\proj$. These arguments can be read backwards, showing that the triviality of the cocycle class in $\h^2(G,K^\times)$ exactly characterises the liftability of $\rho^\proj$ to a linear representation $G \to \GL_n(K)$.

We analyse this a bit further and see that $\sigma|_C$ is a group homomorphism $C \to K^\times$ as $\gamma(c_1,c_2)=1$ for all $c_1,c_2 \in C$.
Moreover, we can replace $\sigma$ by $\sigma \cdot \varphi$ for any character $\varphi: G \to K^\times$.
This leads to twisting $\rho$ by~$\varphi$.
Conversely, if $\sigma': G \to K^\times$ is a map of sets that satisfies the relation $\gamma(g,h)=\sigma'(g)\cdot\sigma'(h)\cdot\sigma'(gh)^{-1}$, then $\sigma \sigma'^{-1}$ is a character, showing that the ambiguity exactly comes from twisting. We summarise this as follows.

\begin{proposition}\label{prop:liftp}
Let $\gamma \in \z^2(\Gbar,K^\times)$ be the $2$-cocycle associated with $\rho^\proj$.
Then the inflation of $\gamma$ is a $2$-coboundary in $\z^2(G,K^\times)$ if and only if $\rho^\proj$ admits a linear lift $\rho:G \to \GL_n(K)$. Different linear lifts are twists of each other by characters $G \to K^\times$.
\end{proposition}

We remark that this can also be elegantly rephrased in the language of group cohomology via the following four terms of the so-called {\em $5$-term exact sequence} (see e.g.\ \cite[(1.6.7)]{NSW})
\[ \Hom(G,K^\times) \xrightarrow{\res} \Hom(C,K^\times) \xrightarrow{\tg} \h^2(\Gbar,K^\times) \xrightarrow{\infl} \h^2(G,K^\times),\]
using that $\h^1 = \Hom$ for trivial modules.
Here $\tg$ is the transgression map, which for a character $\delta:C \to K^\times$ can be explicitly described by
$\tg(\delta)(\pi(g),\pi(h)) = \sigma(g) \cdot \sigma(h) \cdot \sigma(gh)^{-1}$ for $g,h \in G$ and $\sigma: G \to K^\times$ the map defined by $\sigma(c s(\gb)) = \delta(c)$ for $\gb \in \Gbar$ and $c \in C$ where $s: \Gbar \to G$ is any fixed set-theoretic split of~$\pi$ sending $1$ to~$1$.

For Galois representations, we have the following fundamental result of Tate's (see \cite[\S4]{Quer}).

\begin{theorem}\label{thm:Tate} Let $G$ be the absolute Galois group of a global or a local field (non-Archimedean local fields are assumed to have finite residue field) and assume $K$ algebraically closed. Then $\h^{2}(G,K^\times)=1$.
\end{theorem}

Applying the above with an absolute Galois group $G$ as in the theorem and $\Gbar=G$, we see that every projective representation $\rho^\proj: G \to \PGL_n(K)$ admits a (unique up to twisting) lift to a linear representation $\rho: G \to \GL_n(K)$.
Passing to images leads to a central extension
\begin{equation}\label{eq:ce}
1 \to A \to \rho(G) \to \rho^\proj(G) \to 1,
\end{equation}
where $A$ is a subgroup of~$K^\times$, upon identifying $K^\times$ with the group of scalar matrices.
Every central extension corresponds to a $2$-cocycle class (see \cite[(1.2.4)]{NSW}), in this case lying in $\h^2(\rho^\proj(G),A)$. Applying the map
\[ \alpha: \h^2(\rho^\proj(G),A) \to \h^2(\rho^\proj(G),K^\times) \]
induced from the inclusion $A \to K^\times$ returns the $2$-cocycle class attached with~$\rho^\proj$ seen as a representation of its image via inclusion.
We remark explicitly that the map $\alpha$ need not be injective; its kernel is isomorphic to $\Hom(\rho^\proj(G),K^\times/A)/\Hom(\rho^\proj(G),K^\times)$. This means that more than one central extension with kernel~$A$ can be associated with the same projective representation.

Next, we explain how to construct a lift $\rho: G \to \GL_n(K)$ of a given projective representation $\rho^\proj: G \to \PGL_n(K)$ with attached $2$-cocycle class $\gamma \in \h^2(\rho^\proj(G),K^\times)$.
Start with a central extension $\cG$ of $\rho^\proj(G)$ by~$A$, a finite subgroup of~$K^\times$ as in~\eqref{eq:ce}. It corresponds to the class of a cocycle $\delta \in \h^2(\rho^\proj(G),A)$ and we assume $\alpha(\delta)=\gamma$.
We further have by Proposition~\ref{prop:liftp} that the inflation of $\gamma$ to $\h^2(\cG,K^\times)$ is trivial if and only if $\rho^\proj$ can be lifted to a representation $\cG \to \GL_n(K)$.
If $G$ is an absolute Galois group as above, the question to decide whether a map $\pi: G \twoheadrightarrow \cG$ exists such that the diagram
\[\xymatrix@=.5cm{
&&& G \ar@{->>}^{\rho^\proj}[d]\ar@{-->>}_\pi[dl] \\
1 \ar@{->}[r] & A \ar@{->}[r] & \cG \ar@{->}[r] &\rho^\proj(G) \ar@{->}[r] & 1 \\ } \]
commutes, is called a {\em Galois embedding problem}.

In the cases of interest to us, namely the groups $\rho^\proj(G)=\PGL_2(\FF_q)$ with an odd prime power~$q$ and $A = C_{2^r}$, this problem was studied by Quer in~\cite{Quer}, building on work by Serre~\cite{SerreWitt}.
By Propositions 2.1(i) and 2.4~(i)-(ii) in \cite{Quer}, we have the commutative diagram
\begin{equation*}
\xymatrix@=.5cm{
\h^{2}(\PGL_{2}(\FF_{q}),C_{2^{r}}) \ar@{->}^(.55){\sim}[r] \ar@{->}[d]^{\res} & \ZZ/2\ZZ\times\ZZ/2\ZZ  \ar@{->}[d]\\
\h^{2}(\PSL_{2}(\FF_{q}),C_{2^{r}}) \ar@{->}^(.65){\sim}[r]   & \ZZ/2\ZZ. \\
}\end{equation*}
There are thus three non-trivial central extensions of $\PGL_{2}(\FF_{q})$ by~$C_{2^r}$, two of which restrict to the unique non-trivial central extension $2^{r}\PSL_{2}(\FF_{q})$ of $\PSL_{2}(\FF_{q})$ by~$C_{2^r}$. For $r=1$, the latter is simply~$\SL_2(\FF_q)$.
Following \cite[p.~549]{Quer}, the extension $2_{-}\PGL_{2}(\FF_{q})$ is defined by the pull-back of the exact sequence
\[1\rightarrow C_{2} \hookrightarrow \SL_{2}(\Fbar_q)\twoheadrightarrow \PSL_{2}(\Fbar_q)\rightarrow 1\]
under the embedding $\PGL_{2}(\FF_{q})\hookrightarrow \PGL_{2}(\Fbar_q)=\PSL_{2}(\Fbar_q)$.
For every $r>1$, let $2_{-}^{r}\PGL_{2}(\FF_{q})$ be defined as the image of $2_{-}\PGL_{2}(\FF_{q})$ under the map
\[\h^{2}(\PGL_{2}(\FF_{q}),C_{2})\rightarrow \h^{2}(\PGL_{2}(\FF_{q}),C_{2^{r}})\]
induced by the embedding $C_{2}\hookrightarrow C_{2^{r}}$.  
Concretely,
\[2_{-}\PGL_{2}(\FF_{q}) = \{ M \in \SL_2(\Fbar_q) \;|\; \exists\, \lambda \in \Fbar_q^\times: \lambda \cdot M \in \GL_2(\FF_q)\}
= \langle \SL_2(\FF_q), \mat 0 {-y}{y^{-1}} 0 \rangle \]
for any fixed $y \in \FF_{q^2}$ such that $y^2$ is a non-square in $\FF_q^\times$.
As furthermore $C_{2^r}$ lies in the centre, we have 
\[2^r_{-}\PGL_{2}(\FF_{q}) = \langle \SL_2(\FF_q), \mat 0 {-y}{y^{-1}} 0, \mat u00u \rangle \subset \GL_2(\Fbar_q), \]
where $u \in \Fbar_q^\times$ is a fixed element of order $2^r$.

We now assume $G=G_\QQ$ and $q=5$, but point out explicitly that Quer's results cover all~$q$ (but with modified formulas).
We exploit the exceptional isomorphism $S_5 \cong \PGL_2(\FF_5)$.
Let $K_1$ be the fixed field of the stabiliser group of one element for the action of $\rho^\proj(G_\QQ) \cong \PGL_2(\FF_5) \cong S_5$ on $5$ letters and let $d_{K_1}$ be its discriminant. For any prime~$p$, let $w(K_1)_p$ be the Hasse--Witt invariant associated with the trace form $\tr_{K_1/\QQ}(x^2)$ viewed as a quadratic form over~$\QQ_p$, and, further, denote by $(-,-)_{p}$ the Hilbert symbol over~$\QQ_p$.
Now, let $P(K_1)$ be the finite set of primes~$p$ such that $w(K_1)_p \cdot (-2,d_{K_1})_p \neq 1$.
Finally, in order to finish the set-up, let $\mu(p)$ denote the exponent of the highest power of $2$ dividing $p-1$ with the convention $\mu(2)=1$. Set $\mu(\rho^\proj) = \max \{\mu(p) \;|\; p \in P(K_1)\}$ (or $0$ if $P(K_1)=\emptyset$).
We have the following proposition of Quer's (\cite[Prop.~4.1(ii), Theorem~3.7]{Quer}).

\begin{proposition}[Quer]\label{prop:quer}
Let $\rho^\proj: G_\QQ \to \PGL_2(\FF_5)$ be a surjective projective Galois representation.
It has a lifting $\rho: G_\QQ \to \GL_2(\Fbar_5)$ with image $2^r_{-} \PGL_2(\FF_5)$ if and only if $r > \mu(\rho^\proj)$.
\end{proposition}

We apply this now with $K_1 = K_{f_5}$. Then we have $d_{K_1} = 19 \cdot 151$ and the trace form $\tr_{K_{f_5}/\QQ}(x^2)$ is equivalent to the quadratic form $X^{2}_{1}+4X^{2}_{2}-4d_{K_1}X^{2}_{3}+X_{4}X_{5}$ by \cite[Appendix II, Proposition 6]{SerreWitt}.
It can therefore be represented by the matrix $\textrm{diag}(4,-1,-4d_{K_{f_5}},1,1)$ and we compute for its Hasse-Witt invariant at the prime~$p$:
$w(Q_{K_{f_5}})_{p}=(-1,d)_p$ and so $w(Q_{K_{f_5}})_{p} \cdot (-2,d_{K_{f_5}})_p = (2,d_{K_{f_5}})_p = (2,19)_p \cdot (2,151)_p$.
Thus, the only primes that could be in $P(K_{f_5})$ are $2,19,151$. For each of them $\mu(p) = 1$, whence $\mu(\rho^\proj)\le 1$, showing that a lift of $\rho^\proj$ exists for $C_4$.
In fact, a short computation with Hilbert symbols gives us $P(\rho^\proj)=\{19\}$ and so $\mu(\rho^\proj)=1$.

\section{Computational Solution in Characteristic~$5$}

\subsection{An explicit embedding problem}
In this section, we give a concrete computational construction of the lift provided by Proposition~\ref{prop:quer} in our case.
All computations were carried out using Magma~\cite{Magma} and we state the results of these computations here without recalling every time how they were obtained.

In view of the exceptional isomorphism $S_5 \cong \PGL_2(\FF_5)$, the basic idea is to work with a degree~$6$ extension of~$\QQ$ instead of $K_{f_5}$. This is natural because $\PGL_2(\FF_5)$ acts on the six elements of $\PP^1(\FF_5)$.
Concretely, the field $K$, originally defined as the splitting field of $f_5$ over~$\QQ$, is also the splitting field of the polynomial
$g(X)= x^6 - x^5 - 10x^4 + 30x^3 - 31x^2 + 7x + 9 \in \ZZ[x]$. We let $K_g := \QQ[x]/(g(x))$\footnote{\cite[Number Field 6.2.23615200909.1]{lmfdb}}
and consider the projective Galois representation
\[\rho^\proj: G_\QQ \twoheadrightarrow \Gal(K/\QQ) \cong \PGL_2(\FF_5). \]
We know from \S\ref{sec:gp} that there is a linear lift with kernel~$C_4$.
Now, we will construct a polynomial the splitting field of which corresponds to such a lift.
By \S\ref{sec:gp}, there are three non-trivial group extensions of $\PGL_2(\FF_5)$ by~$C_4$, but only two of them restrict to the unique non-split extension of $A_5=\PSL_2(\FF_5)$ by~$C_4$. The split extension of $A_5$ by $C_4$ cannot correspond to a linear lift of the projective representation.
The other two extensions are $\GL_2(\FF_5)$ and $C_4._6 S_5$ in the notation of Tim Dokchitser's project GroupNames~\cite{GroupNames}.
The former cannot occur either since in that case the determinant of $\rho$ would be a Dirichlet character of order~$4$ ramifying only at $19$ and $151$, which does not exist because $(\ZZ/19 \cdot 151\ZZ)^\times$ does not possess any element of order~$4$. Consequently, $C_4._6 S_5$ is the extension $4_{-} \PGL_2(\FF_5)$ (this can also be verified explicitly).

The group $\cG=C_4._6 S_5$ is a transitive permutation group on $48$ letters, and this is the minimum.
One finds that $\cG$ has a unique conjugacy class of subgroups $[H]$ of order~$80$.
Furthermore, $H$ contains a unique normal subgroup $U$ of order~$5$. The quotient $H/U$ is isomorphic to $C_8 \times C_2$.
There are hence also two normal subgroups $N_1,N_2$ of~$H$ of order~$10$ having $C_8$ as quotient. None of them contains a non-trivial normal subgroup of~$G$.

If there is a Galois extension $\tilde{K}$ of~$\QQ$ with Galois group $\cG$ such that $K_g = \tilde{K}^H$, then by the preceding group theory discussion, $K_g$ admits two cyclic extensions of degree~$8$ contained in~$\tilde{K}$ and both these extensions have $\tilde{K}$ as splitting field. This necessary condition leads us to look for $C_8$-extensions of~$K_g$ in order to construct~$\tilde{K}$.
One can use explicit class field theory in Magma to find a cyclic extension of degree~$8$ of $K_g$ inside the ray class field of conductor $151$ if one allows one of the two infinite places to ramify. One can compute a polynomial $h \in \ZZ[x]$ of degree~$48$ describing it and computationally check that its Galois group is indeed~$\cG$. See the appendix \S\ref{subsec:h} for an example of such a polynomial.
We remark that it is not enough to include only one of the two primes above~$151$ into the conductor.
This is in accordance with the computations at the inertia groups at $19$ and $151$ below.

\subsection{A linear Galois representation}\label{ssec:lingal}

By the explicit matrix description of $4_{-} \PGL_2(\FF_5) \cong C_4._6 S_5$ given in \S\ref{sec:gp}, we obtain a Galois representation
\[\rho: G_\QQ \twoheadrightarrow \Gal(\tilde{K}/\QQ)=\cG \subset \GL_2(\Fbar_5) \]
lifting $\rho^\proj$ with image the subgroup of $\GL_2(\Fbar_5)$ generated by $\SL_2(\FF_5)$, the scalar $\mat 2002$ and the order~$2$ matrix $\mat 0 \zeta {-\zeta^{-1}} 0$, where we take $\zeta \in \FF_{5^2}^\times$ of order $8$ satisfying $\zeta^2=2$.
This explicit description allows us to relate the cycle type of an element in $S_5 \cong \PGL_2(\FF_5)$ to the trace and determinant of all possible lifts. Table~\ref{table:tr} contains all pairs of trace and determinant that occur for a given cycle type as well as other information.

\begin{table}
\begin{center}
\begin{tabular}{||l|l|l|l||}
\hline
Conjugacy class in $S_5$ & $\tr$ in $\theta_5^\st$ & $(\tr,\det)$ in $\cG \subset \GL_2(\Fbar_5)$ &  $\sgn$ \\
\hline
$(1)            $  &  $4$  &     $ (1, 4), (2, 1), (3, 1), (4, 4)$  & $ 1$ \\
$(1, 3, 5, 4, 2)$  &  $-1$ &     $ (1, 4), (2, 1), (3, 1), (4, 4)$  &  $ 1$ \\
$(2, 5)(3, 4)   $  &  $0$  &     $ (0, 1), (0, 4)$                  & $ 1$ \\
$(1, 4)         $  &  $2$  &     $ (0, 1), (0, 4)$                  &  $-1$ \\
$(1, 4, 5)      $  &  $1$  &     $ (1, 1), (2, 4), (3, 4), (4, 1) $ &  $ 1$ \\
$(1, 5)(2, 3, 4)$  &  $-1$ &     $ (\zeta, 4), (2 \zeta, 1), (-\zeta, 4), (-2\zeta, 1)$ &  $-1$ \\
$(1, 2, 5, 3)   $  &  $0$  &     $ (\zeta, 1), (2 \zeta, 4), (-\zeta, 1), (-2\zeta, 4)$ &  $-1$ \\
\hline
\end{tabular}
\end{center}
\caption{Data on representations}\label{table:tr}
\end{table}

We next determine the conductor of~$\rho$. As only $19$ and $151$ ramify, the ramification is tame and inertia groups are cyclic.
Recall that at both primes the inertia groups in $K/\QQ$ are of order~$2$ generated by transpositions. Each one of the corresponding inertia groups of $\tilde{K}/\QQ$ will hence be generated by a lift of a transposition. According to Table~\ref{table:tr}, such lifts have characteristic polynomials $X^2-1$ or $X^2+1$ and thus the inertia orders are $2$ or~$4$.
Recall further that the polynomial~$h$ was obtained via a ray class field unramified at~$19$. Consequently, the order of inertia at~$19$ in $\tilde{K}/\QQ$ is still~$2$ and, moreover, it fixes a line since $1$ is an eigenvalue of the inertia generator.
As the extension $\tilde{K}/K$ ramifies at~$151$, the inertia group at $151$ of $\tilde{K}/\QQ$ is of order~$4$ and does not fix any line.
This implies that the conductor of~$\rho$ is $19 \cdot 151^2=433219$.

The group $\cG$ admits three surjective group homomorphisms $\cG \to C_2$, namely: the determinant $\det$ (via the embedding of $\cG$ in $ \GL_2(\Fbar_p)$ described above), the sign of a permutation $\sgn$ via the projection $\cG \to S_5$ and the product $\det \cdot \sgn$.
As $\QQ(\sqrt{19\cdot 151})$ is fixed by the sign, we have $\sgn = \chi_{19 \cdot 151} = \left(\frac{19 \cdot 151}{\cdot}\right)$. As the characteristic polynomial of a generator of the inertia group of $151$ is $X^2+1$, the character $\det \circ \rho$ is unramified at~$151$; it does ramify at~$19$. Consequently, $\det \circ \rho = \chi_{-19} = \left(\frac{-19}{\cdot}\right)$ and $\det \circ \sgn = \left(\frac{-151}{\cdot}\right)$.
Table~\ref{table:ker} summarises this information and names the three normal subgroups of~$\cG$ of index~$2$.

\begin{table}
\begin{center}
\begin{tabular}{||l|l|l|l||}
\hline
Character          & Quadratic Field & Group Name \cite{GroupNames} of $\ker$ & Generators \\
\hline
$\sgn$             & $E=\QQ(\sqrt{19\cdot 151})$ & $C_4 . A_5$ & $\SL_2(\FF_5)$, $\mat 2002$ \\
$\det$             & $\QQ(\sqrt{-19})$ & $\mathrm{CSU}_2(\FF_5) \cong 2_{-}\PGL_2(\FF_5)$ & $\SL_2(\FF_5)$, $\mat 0{\zeta}{-\zeta^{-1}}0$  \\
$\det\cdot \sgn$   & $\QQ(\sqrt{-151})$ & $C_2.S_5$ & $\SL_2(\FF_5)$, $\mat 0{\zeta}{\zeta^{-1}}0$  \\
\hline
\end{tabular}
\end{center}
\caption{Normal subgroups of index $2$ in $\cG$}\label{table:ker}
\end{table}

Let $L = \QQ(\sqrt{-19},\sqrt{-151}) \subset \tilde{K}$ be the compositum of the three corresponding quadratic fields.
We first remark that $L/E$ is an unramified CM extension. It hence corresponds to a quadratic character
$\epsilon: G_E \to \{\pm 1\} \subset \FF_5^\times$, which is unramified at all finite places and totally odd.

Another character will be of importance to us. Let $\fp$ be the prime of~$E$ lying above~$151$.
The ray class group of $E$ of conductor $\fp \infty_1$ is cyclic of order~$150$. Thus, $E$ admits a $C_2$-extension ramifying only at~$\fp$ and one of the two (real) places. Let $\delta: G_E \to \{\pm 1\} \subset \FF_5^\times$ be the corresponding character. It is not the restriction of any character of~$G_\QQ$.

The restriction to $G_E$ of $\rho$ is unramified at~$19$ (as $I(\tilde{K}/\QQ)_{19}=C_2$ and $E/\QQ$ ramifies at~$19$), but it does ramify at~$151$. The inertia group $I(\tilde{K}/\QQ)_{151}$ is generated by an order~$4$ matrix of determinant~$1$ lifting a transposition, whence it is conjugate to $\mat 0{2\zeta}{2\zeta^{-1}}0$, so that $I(\tilde{K}/E)_{151}$ is generated by its square, i.e.\ by $\mat {-1}00{-1}$.
Consequently, the twist $\rho|_E \otimes \delta$ is unramified at all finite places.
It is a lift of the projective representation $G_E \to A_5$.

\begin{proof}[Proof of Theorem~\ref{thm:F}.]
Let $\rho$ be the Galois representation constructed in this section.
By Serre's Modularity Conjecture proved in \cite[Theorem~1.2]{KW1} and \cite[Corollary~0.2]{Kisin}, together with results on the optimal weight due to Edixhoven~\cite[Theorem~4.5]{Edix}, there exists a Hecke eigenform $F \in S_1(19\cdot 151^2,\chi_{-19},\Fbar_5)$ such that its attached Galois representation $\rho_F$ is isomorphic to~$\rho$.
The other assertions have been established above except for the non-liftability to a holomorphic weight one modular form. This simply follows from the well-known group theoretic result already known to Klein~\cite{Klein} that $S_5$ is not a subquotient of~$\PGL_2(\CC)$, contradicting the existence of any attached Artin representation.
\end{proof}

As its level is very big, we do not see how to compute the weight one modular form~$F$ explicitly on the computer without using its Galois representation~$\rho$.

\begin{proof}[Proof of Corollary~\ref{cor:int2}.]
All statements can be verified using Tables~\ref{table:tr} and~\ref{table:ker} together with the relation $a_p(F) = \tr(\rho(\Frob_p))$ and
$\sgn = \left(\frac{19 \cdot 151}{\cdot}\right)$ as well as $\det \circ \rho = \chi_{-19} = \left(\frac{-19}{\cdot}\right)$.
More conceptually, the congruence of $\theta_5^\st$ can also be derived from Corollary~\ref{cor:Asai}.
\end{proof}

\subsection{Appendix}\label{subsec:h}
Here is a polynomial the splitting field of which is the field cut out by $\rho$ in characteristic~$5$ from Theorem~\ref{thm:F}.

\begin{small}
\noindent
$h(X) = x^{48} - 10 x^{47} - 13 x^{46} + 173 x^{45} - 1278 x^{44} + 27542 x^{43} - 113958 x^{42} - 286430 x^{41} + 4655329 x^{40}\\
- 26503188 x^{39} + 81919958 x^{38} + 32368110 x^{37} - 2439071195 x^{36} +  10669493052 x^{35}- 26002615844 x^{34}\\
+ 164051953843 x^{33} - 205565265490 x^{32} - 3098320327510 x^{31} + 15580543347067 x^{30}- 72094759904784 x^{29}\\
+ 145352373756651 x^{28} + 1124294833301773 x^{27} - 4736762045102396 x^{26}- 4428623245164253 x^{25}\\
+ 46182217850444449 x^{24} - 135621698076328862 x^{23}+ 69305601476994468 x^{22} + 3791910125162463418 x^{21}\\
- 14065814910470191337 x^{20} - 13348365591179322148 x^{19} + 124088837951469551773 x^{18}\\
- 286160102141567453230 x^{17}+ 886712293571081863675 x^{16} + 1149044936598536032213 x^{15}\\
- 14719660664892430787424 x^{14}+ 10532624944253653528232 x^{13} + 56786830275191356552239 x^{12}\\
- 52406153009314731797162 x^{11} - 149323467251503445783614 x^{10} +  669256616167712724103315 x^{9}\\
- 899500431661959205787756 x^{8} - 3108487402346193671659483 x^{7} + 4134225816838771492997125 x^{6}\\
+ 14451282311965453942468438 x^{5} -     6338226206230170122590826 x^{4} - 39455974427388666679528925 x^{3}\\
- 30466901209941980350644125 x^{2} + 70704214646412544819950625 x + 72894568328135627845675625$
\end{small}

\section{Solution via Asai transfer}

\subsection{The standard representation via the Asai transfer}

In this section we work with complex representations.
For the convenience of the reader, we recall the construction of the Asai transfer (also called tensor induction or multiplicative induction) of a group representation. We follow~\cite{Prasad}.  
Let $G$ be a group and $H$ a subgroup of~$G$ of index~$m$. Let $V$ be an $n$-dimensional representation of $H$. Let $g_{1},\ldots,g_{m}$ be a set of representatives for the left cosets of $H$ in $G$. For $g\in G$ and for each $j\in\{1,\ldots,m\}$, choose $i\in \{1,\dots,m\}$ such that $gg_{i}\in g_{j}H$ and define $h(g,i)\in H$ by $gg_{i}=g_{j}h(g,i)$. The Asai transfer of $V$ from $H$ to $G$, denoted $\text{Asai}^{G}_{H}(V)$, is the vector space $V^{\otimes m}$ equipped with the action defined by 
\[ g(v_{1}\otimes\ldots\otimes v_{m})=w_{1}\otimes\ldots\otimes w_{m}\]
where, for each $j\in\{1,\ldots,m\}$, $w_{j}=h(g,i)v_{i}$.  

We now describe the special case of tensor induction which we will need.
We assume the index of $H$ in~$G$ to be~$2$ and we let $\eta:H \to \GL_n(\CC)$ be a representation with character~$\psi$.
For $g \in G \setminus H$ and $h \in H$, we then have (see Lemma 4.1 of \cite{Isaacs} and the discussion preceding it)
\begin{equation}\label{eq:asai}
 \tr(\Asai_H^G(\eta)(h)) = \psi(h) \psi(g^{-1}hg) \textnormal{ and } \tr(\Asai_H^G(\eta)(g)) = \psi(g^2).
\end{equation}

Let $r \ge 1$ and $\eta: 2^r\PSL_2(\FF_5) \to \GL_2(\CC)$ be an irreducible representation.
Write $\Asai(\eta)$ for $\Asai_{2^r\PSL_2(\FF_5)}^{2^r_{-}\PGL_2(\FF_5)}(\eta)$.
We now describe it on any element~$c$ in the centre of $2^r_{-}\PGL_2(\FF_5)$. Such $c$ lies in $2^r\PSL_2(\FF_5)$ and we have $\eta(c) = \mat \lambda 00 \lambda$. Consequently, from \eqref{eq:asai} we get
\begin{equation}\label{eq:asai1}
  \tr\big(\Asai(\eta)(c)\big) = 4 \lambda^2 = 4 \cdot \det(\eta(c)).
\end{equation}
We next aim at twisting the representation appropriately, making it trivial on the centre.

\begin{lemma}\label{lem:asai}
For $r \ge 1$ let $\alpha: C_{2^r} \to \CC^\times$ and $\beta: \PGL_2(\FF_5) \to \CC^\times$ be characters such that $\alpha$ restricted to the subgroup $C_2$ of $C_{2^r}$ is trivial.
Then there exists a unique character
\[ \chi: 2^r_{-}\PGL_2(\FF_5) \to \CC^\times \]
such that $\chi|_{C_{2^r}} = \alpha$ and $\chi|_{2_{-}\PGL_2(\FF_5)}=\beta \circ \pi$ for the natural projection $\pi: 2_{-}\PGL_2(\FF_5) \twoheadrightarrow \PGL_2(\FF_5)$.
\end{lemma}

\begin{proof}
The point is that the image of the $2$-cycle $\gamma \in \h^2(\PSL_2(\FF_5),C_{2^r})$ which describes the central extension $2^r_{-}\PGL_2(\FF_5)$ lies in $C_2$ by construction.
Writing elements of $2^r_{-}\PGL_2(\FF_5)$ uniquely as $(c,g) \in C_{2^r} \times \PGL_2(\FF_5)$, we define $\chi$ uniquely by letting $\chi\big((c,g)\big) = \alpha(c)\beta(g)$.
This is indeed a group homomorphism with the desired properties because
\[ \chi\big( (c,g) \cdot (c',g') \big) = \chi\big( (cc'\gamma(g,g'), gg' ) \big) = \alpha(cc'\gamma(g,g')) \cdot \beta(gg') = \chi\big( (c,g) \big) \cdot\chi\big( (c',g') \big) \]
since $\alpha(\gamma(g,g')) = 1$ by assumption.
\end{proof}

\begin{proposition}\label{prop:asai}
Let $r \ge 1$, $\eta: 2^r\PSL_2(\FF_5) \to \GL_2(\CC)$ and $\Asai(\eta)$ as above.
Let \newline $\chi: 2^r_{-}\PGL_2(\FF_5) \to \CC^\times$ be the unique character from Lemma~\ref{lem:asai} such that $\chi|_{C_{2^r}} = \big(\det \circ \eta|_{C_{2^r}}\big)^{-1}$ and $\chi|_{2_{-}\PGL_2(\FF_5)}=\epsilon \circ \pi$ where $\epsilon: \PGL_2(\FF_5) \cong S_5 \to \{\pm 1\}$ is the sign character.

Then $\Asai(\eta) \otimes \chi$ factors through $\PGL_2(\FF_5) \cong S_5$ and \[\Asai(\eta) \otimes \chi \cong \theta_5^\st.\]
\end{proposition}

\begin{proof}
By \eqref{eq:asai1}, the restriction of $\Asai(\eta) \otimes \chi$ to $C_{2^r}$ is the trivial $4$-dimensional representation, implying that it factors through $\PGL_2(\FF_5) \cong S_5$.
An inspection of the character table of $S_5$ shows that $\Asai(\eta) \otimes \chi$ is then one of the two irreducible $4$-dimensional representations of~$S_5$, which are $\theta_5^\st$ and $\theta_5^\st \otimes \epsilon$.
Indeed, if it were a sum of four $1$-dimensional representations, then all character values would be even, which is not the case as the trace of $\eta$ is odd on elements of order~$3$ in $2\PSL_2(\FF_5)$.

As in \S\ref{ssec:lingal}, consider $g = \mat 0 {-\zeta} {\zeta^{-1}} 0 \in 2_{-}\PGL_2(\FF_5)$ for $\zeta \in \FF_{5^2}$ such that $\zeta^2=2$ is a non-square in~$\FF_5$. We have $\tr(\eta(g^2)) = \tr(\eta(\mat{-1}00{-1})) = -2$. As $g$ lies in $2_{-}\PGL_2(\FF_5)$ but not in $2\PSL_2(\FF_5)$ its projection to $\PGL_2(\FF_5) \cong S_5$ is a transposition, whence $\chi(g) = \epsilon(g) = -1$ and $\tr(\rho^\st(g)) = 2$.
Thus, this computation proves that $\Asai(\eta) \otimes \chi$ is not isomorphic to $\theta_5^\st \otimes \epsilon$, so it is isomorphic to $\theta_5^\st$.
\end{proof}

In view of \eqref{eq:asai}, we obtain the following description of the character of $\theta_5^\st$.
\begin{corollary}\label{cor:asai}
With notation as in Proposition~\ref{prop:asai} and $\psi = \tr \circ \eta$, for any $g \in S_5 \setminus A_5$ and any $h \in A_5$ we have
\[ \tr(\theta_5^\st(h)) = \psi(\hat{h})\cdot \psi(\hat{g}^{-1}\hat{h}\hat{g}) \cdot \chi(\hat{h})
\textnormal{ and }
   \tr(\theta_5^\st(g)) = \psi(\hat{g}^2) \cdot \chi(\hat{g}), \]
where $\hat{g} \in 2^r_{-}\PGL_2(\FF_5)$ and $\hat{h} \in 2^r\PSL_2(\FF_5)$ are any lifts of $g$ and $h$, respectively.
\end{corollary}

\subsection{Automorphy}

In this section, we prove Theorem~\ref{thm:Asai}. The key input providing the automorphy is the following strong result of Pilloni and Stroh.

\begin{theorem}[\cite{Pil}, Th\'eor\`eme 0.3]\label{thm:PS}
Let $E$ be a totally real field and $\eta:G_{E}\rightarrow \GL_{2}(\CC)$ be a totally odd, irreducible representation.
Then $\eta$ is modular, attached to a Hilbert cuspidal eigenform of weight one. 
\end{theorem}

\begin{proof}[Proof of Theorem~\ref{thm:Asai}~\eqref{thm:Asai:a}.]
We start by viewing the $S_5$-extension $K/\QQ$ as a surjective projective Galois representation $\rho^\proj: G_\QQ \to \PGL_2(\FF_5)$.
By Proposition~\ref{prop:quer}, it lifts to a linear Galois representation $\rho: G_\QQ \to \GL_2(\Fbar_5)$ with image $2^r_{-}\PGL_2(\FF_5)$ for any fixed choice of $r > \mu(\rho^\proj)$.
Let $\tilde{K}$ be the number field `cut out' by~$\rho$, {\it i.e.\ }the one such that its absolute Galois group equals $\ker(\rho)$.
Then $\Gal(\tilde{K}/\QQ) \cong 2^r_{-}\PGL_2(\FF_5)$, the subgroup $\Gal(\tilde{K}/K)$ is its centre $C_{2^r}$ and
$\Gal(\tilde{K}/E) \cong 2^r\PSL_2(\FF_5)$.

Let now
\[\eta: G_E \twoheadrightarrow G(\tilde{K}/E) \cong 2^r\PSL_2(\FF_5) \to \GL_2(\CC) \]
be obtained from any two-dimensional irreducible complex representation of $2^r\PSL_2(\FF_5)$ (such a representation exists because $2\PSL_2(\FF_5)$ admits two of them and the centre can be realised via scalar matrices).
Let $c\in G_{E}$ be any complex conjugation. As $K$ is totally imaginary, $c$ does not lie in the centre of $2^r\PSL_2(\FF_5)$.
Thus $\eta(c)$ is a non-scalar involution in $\GL_2(\CC)$ and as such has determinant~$1$.
Consequently, $\eta$ is a totally odd representation. 
Then Theorem~\ref{thm:PS} shows the existence of the claimed Hilbert modular form~$G$.

Seeing $\eta$ alternatively as a representation of $\Gal(\tilde{K}/E)$, we naturally identify $\Asai_{G_E}^{G_\QQ}(\eta)$ with $\Asai_{\Gal(\tilde{K}/E)}^{\Gal(\tilde{K}/\QQ)}(\eta)$. The claimed formula is now the content of Proposition~\ref{prop:asai}.
\end{proof}

We next appeal to the functoriality of the Asai transfer.  
Let $L/F$ be a quadratic extension of number fields and $\pi=\bigotimes_{w}\pi_{w}$ be a cuspidal representation of $GL_{2}(\mathbb{A}_{L})$. If $\rho:G_L\rightarrow \GL_{2}(\CC)$ is a Galois representation such that its Artin $L$-function equals $L(s,\pi)$, except for finitely many places, one can associate an $L$-function to $\pi$, denoted $L_{\Asai}(s,\pi)$, in such a way that the local factors of $L_{\Asai}(s,\pi)$ match the local factors of the Artin $L$-function of $\Asai_{G_L}^{G_F}(\rho)$, again, with the exception of finitely many places. We refer the readers to the articles \cite{Ramakrishnan} and to sections 2 and 3 of \cite{Krishna} for the relevant constructions and for the proof of the following result. 
 
\begin{theorem}[{\cite[Theorem 1.4~(a)]{Ramakrishnan}}]\label{thm:rama}
Let $L/F$ be a quadratic extension of number fields, and let $\pi$ be a cuspidal automorphic representation of $\GL_{2}(\AAA_{L})$. Then there exists an automorphic representation $\Pi$ for $\GL_{4}(\AAA_{F})$ such that the $L$-function of $\Pi$ equals $L_{\Asai}(s, \pi)$ except at finitely many finite places. We denote by $\Asai(\pi)$ the automorphic form $\Pi$.
\end{theorem}

\begin{proof}[Proof of Theorem~\ref{thm:Asai}~\eqref{thm:Asai:b}.]
The Galois representation $\eta$ is attached to a cuspidal automorphic representation for $GL_{2}(\mathbb{A}_{E})$, say $\pi$, corresponding to the Hilbert modular form~$G$.
By Theorem~\ref{thm:rama} applied with $L=E$ and $F=\QQ$, we obtain that the L-function of $\Asai(\pi)$ equals the Artin L-function of $\Asai_{G_E}^{G_\QQ}(\eta)$.
The result follows by twisting $\Asai(\pi)$ by the Hecke character corresponding to~$\chi$ because that twist corresponds to twisting $\Asai_{G_E}^{G_\QQ}(\eta)$ by~$\chi$.
\end{proof}

\begin{proof}[Proof of Corollary~\ref{cor:Asai}.]
We specialise Theorem~\ref{thm:Asai}~\eqref{thm:Asai:a} to the splitting field $K$ of~$f_5$ over~$\QQ$.
Table~\ref{table:ker} shows that $\chi = \left(\frac{- 151}{\cdot}\right)$ because $\chi|_{C_4} = \det \circ \eta$ and $\chi|_{2_{-}\PGL_2(\FF_5)}$ factors through $\PGL_2(\FF_5) \cong S_5$ as the sign character.
Furthermore, if $\psi$ denotes the character of~$\eta$, by the properties of $\eta$, for any unramified finite place $\fp$ of~$E$ we have $\psi(\fp) = a_\fp(G)$. The proof is now finished by Corollary~\ref{cor:asai} and an inspection of Table~\ref{table:tr}.
\end{proof}

\bibliography{References}

\begin{thebibliography}{{LMF}22}

\bibitem[BCP97]{Magma}
Wieb Bosma, John Cannon, and Catherine Playoust.
\newblock The {M}agma algebra system. {I}. {T}he user language.
\newblock volume~24, pages 235--265. 1997.
\newblock Computational algebra and number theory (London, 1993).

\bibitem[BMT15]{BMT}
David~W. Boyd, Greg Martin, and Mark Thom.
\newblock Squarefree values of trinomial discriminants.
\newblock {\em LMS J. Comput. Math.}, 18(1):148--169, 2015.

\bibitem[Bry12]{Bryk}
John~T. Bryk.
\newblock {\em On the roots of polynomials modulo primes}.
\newblock ProQuest LLC, Ann Arbor, MI, 2012.
\newblock Thesis (Ph.D.)--Rutgers The State University of New Jersey - New
  Brunswick.

\bibitem[Cal13]{Cal}
Frank Calegari.
\newblock The {A}rtin conjecture for some {$S_5$}-extensions.
\newblock {\em Math. Ann.}, 356(1):191--207, 2013.

\bibitem[Dok]{GroupNames}
Tim Dokchitser.
\newblock groupnames.org.

\bibitem[Dwy14]{Dwyer}
J.~Dwyer.
\newblock Real zeros of {A}rtin {$L$}-functions corresponding to
  five-dimensional {$S_5$}-representations.
\newblock {\em Bull. Lond. Math. Soc.}, 46(1):51--58, 2014.

\bibitem[Edi92]{Edix}
Bas Edixhoven.
\newblock The weight in {S}erre's conjectures on modular forms.
\newblock {\em Invent. Math.}, 109(3):563--594, 1992.

\bibitem[Edi97]{EdixSC}
Bas Edixhoven.
\newblock Serre's conjecture.
\newblock In {\em Modular forms and {F}ermat's last theorem ({B}oston, {MA},
  1995)}, pages 209--242. Springer, New York, 1997.

\bibitem[Gor02]{Goren}
Eyal~Z. Goren.
\newblock {\em Lectures on {H}ilbert modular varieties and modular forms},
  volume~14 of {\em CRM Monograph Series}.
\newblock American Mathematical Society, Providence, RI, 2002.
\newblock With the assistance of Marc-Hubert Nicole.

\bibitem[Isa82]{Isaacs}
I.~M. Isaacs.
\newblock Character correspondences in solvable groups.
\newblock {\em Adv. in Math.}, 43(3):284--306, 1982.

\bibitem[Kat73]{KatzPadicProp}
Nicholas~M. Katz.
\newblock {$p$}-adic properties of modular schemes and modular forms.
\newblock In {\em Modular functions of one variable, {III} ({P}roc. {I}nternat.
  {S}ummer {S}chool, {U}niv. {A}ntwerp, {A}ntwerp, 1972)}, Lecture Notes in
  Math., Vol. 350, pages 69--190. Springer, Berlin, 1973.

\bibitem[Ked12]{Kedlaya}
Kiran~S. Kedlaya.
\newblock A construction of polynomials with squarefree discriminants.
\newblock {\em Proc. Amer. Math. Soc.}, 140(9):3025--3033, 2012.

\bibitem[Kis09]{Kisin}
Mark Kisin.
\newblock Modularity of 2-adic {B}arsotti-{T}ate representations.
\newblock {\em Invent. Math.}, 178(3):587--634, 2009.

\bibitem[Kle93]{Klein}
Felix Klein.
\newblock {\em Vorlesungen {\"u}ber das {Ikosaeder} und die {Aufl{\"o}sung} der
  {Gleichungen} vom f{\"u}nften {Grade}. {Hrsg}. mit einer {Einf{\"u}hrung} und
  mit {Kommentaren} von {Peter} {Slodowy}}.
\newblock Basel: Birkh{\"a}user Verlag; Stuttgart: B. G. Teubner
  Verlagsgesellschaft, reprogr. {Nachdr}. d. {Ausg}. {Leipzig} 1884, {Teubner}
  edition, 1993.

\bibitem[Kon95]{Kondo}
Takeshi Kondo.
\newblock Algebraic number fields with the discriminant equal to that of a
  quadratic number field.
\newblock {\em J. Math. Soc. Japan}, 47(1):31--36, 1995.

\bibitem[Kri12]{Krishna}
M.~Krishnamurthy.
\newblock Determination of cusp forms on {$GL(2)$} by coefficients restricted
  to quadratic subfields (with an appendix by {D}ipendra {P}rasad and {D}inakar
  {R}amakrishnan).
\newblock {\em J. Number Theory}, 132(6):1359--1384, 2012.

\bibitem[KW09]{KW1}
Chandrashekhar Khare and Jean-Pierre Wintenberger.
\newblock Serre's modularity conjecture. {I}.
\newblock {\em Invent. Math.}, 178(3):485--504, 2009.

\bibitem[{LMF}22]{lmfdb}
The {LMFDB Collaboration}.
\newblock The {L}-functions and modular forms database.
\newblock \url{http://www.lmfdb.org}, 2022.

\bibitem[LNV91]{LNV2}
P.~Llorente, E.~Nart, and N.~Vila.
\newblock Decomposition of primes in number fields defined by trinomials.
\newblock {\em S\'{e}m. Th\'{e}or. Nombres Bordeaux (2)}, 3(1):27--41, 1991.

\bibitem[NSW08]{NSW}
J\"{u}rgen Neukirch, Alexander Schmidt, and Kay Wingberg.
\newblock {\em Cohomology of number fields}, volume 323 of {\em Grundlehren der
  mathematischen Wissenschaften [Fundamental Principles of Mathematical
  Sciences]}.
\newblock Springer-Verlag, Berlin, second edition, 2008.

\bibitem[Pra92]{Prasad}
Dipendra Prasad.
\newblock Invariant forms for representations of {${\rm GL}_2$} over a local
  field.
\newblock {\em Amer. J. Math.}, 114(6):1317--1363, 1992.

\bibitem[PS16]{Pil}
Vincent Pilloni and Beno\^{\i}t Stroh.
\newblock Surconvergence, ramification et modularit\'{e}.
\newblock {\em Ast\'{e}risque}, (382):195--266, 2016.

\bibitem[Que95]{Quer}
Jordi Quer.
\newblock Liftings of projective {$2$}-dimensional {G}alois representations and
  embedding problems.
\newblock {\em J. Algebra}, 171(2):541--566, 1995.

\bibitem[Ram02]{Ramakrishnan}
Dinakar Ramakrishnan.
\newblock Modularity of solvable {A}rtin representations of {${\rm
  GO}(4)$}-type.
\newblock {\em Int. Math. Res. Not.}, (1):1--54, 2002.

\bibitem[Sel56]{Selmer}
Ernst~S. Selmer.
\newblock On the irreducibility of certain trinomials.
\newblock {\em Math. Scand.}, 4:287--302, 1956.

\bibitem[Ser84]{SerreWitt}
Jean-Pierre Serre.
\newblock L'invariant de {W}itt de la forme {${\rm Tr}(x^2)$}.
\newblock {\em Comment. Math. Helv.}, 59(4):651--676, 1984.

\bibitem[Ser03]{Jordan}
Jean-Pierre Serre.
\newblock On a theorem of {J}ordan.
\newblock {\em Bull. Amer. Math. Soc. (N.S.)}, 40(4):429--440, 2003.

\end{thebibliography}

\bibliographystyle{alpha}

\end{document}